\documentclass[12pt]{amsart}
\usepackage[margin=1.2in]{geometry}
\usepackage{tikz,esvect}
\usepackage{pgfplots}

\usepackage{amsmath,amssymb,amsfonts,amsthm,mathtools}
\usepackage{enumitem}
\usepackage{hyperref}
\usepackage{verbatim}
\usepackage{xcolor}

\hypersetup{
  colorlinks=true,
  linkcolor=blue,
  citecolor=blue,
  urlcolor=blue
}

\theoremstyle{plain}
\newtheorem{theorem}{Theorem}[section]
\newtheorem{lemma}[theorem]{Lemma}
\newtheorem*{lemma*}{Lemma} 

\newtheorem{corollary}[theorem]{Corollary}
\newtheorem{openproblem}[theorem]{Open Problem}

\theoremstyle{definition}

\theoremstyle{remark}
\newtheorem{remark}[theorem]{Remark}

\newcommand{\R}{\mathbb{R}}

\newcommand{\T}{\mathbb{T}}
\newcommand{\Z}{\mathbb{Z}}

\newcommand{\ep}{\varepsilon}

\begin{document}

\title[Nonexistence of vanishing-viscosity limits]{Nonexistence of vanishing-viscosity limits for mechanical Hamiltonian ergodic problems}

\author[Z. LIU, H. V. TRAN, Y. YU]
{Ziran Liu, Hung V. Tran, Yifeng Yu}

\dedicatory{Dedicated to the memory of Michael G. Crandall with our deepest respect and admiration}

\thanks{
H. V. Tran is partially supported by NSF grant DMS-2348305. 
}

\address[Z. Liu]
{Shanghai Institute for Mathematics and Interdisciplinary Sciences (SIMIS), Shanghai, China, 200433}
\address[Z. Liu]
{Research Institute of Intelligent Complex Systems, Fudan University, Shanghai 200433, China}
\email{zliu@simis.cn}

\address[H. V. Tran]
{
Department of Mathematics, 
University of Wisconsin-Madison, Van Vleck Hall, 480 Lincoln Drive, Madison, Wisconsin 53706, USA}
\email{hung@math.wisc.edu}

\address[Y. Yu]
{
Department of Mathematics, 
University of California at Irvine, 
California 92697, USA}
\email{yifengy@uci.edu}


\begin{abstract}
For \(\varepsilon>0\), let \(\phi^\varepsilon\) be the solution of the ergodic problem
\[
    \frac12 |D\phi^\varepsilon|^2+F(x)-\varepsilon\Delta\phi^\varepsilon=c(\varepsilon)
    \qquad \text{on } \T^n,
\]
normalized by \(\phi^\varepsilon(0)=0\). 
We construct a one-dimensional example with \(F\in C^3\) for which the vanishing-viscosity limit $\lim_{\varepsilon\to0}\phi^\varepsilon$
does not exist. This gives a negative answer to a problem proposed by Jauslin, Kreiss, and Moser \cite{JKM}.
\end{abstract}

\maketitle
\section{Introduction}
\subsection{The selection problem}
Let $\T^n=\R^n/\Z^n$ be the flat $n$-dimensional torus.
We consider the mechanical Hamiltonian
\[
    H(x,p)=\frac12 |p|^2+F(x),
\]
where $F\in C^3(\T^n)$.
For every $\ep>0$, consider the ergodic (cell) problem
\begin{equation}\label{eq:viscousHJ}
    \frac12 |D\phi^\ep|^2+F(x)-\ep\Delta\phi^\ep=c(\ep)
    \qquad\text{on }\T^n.
\end{equation}
There exists a unique constant $c(\ep)\in \R$ such that \eqref{eq:viscousHJ} has a solution.
In fact, denote by
\begin{equation}\label{eq:HopfCole}
    u^\ep=\exp\left(-\frac{\phi^\ep}{2\ep}\right)>0.
\end{equation}
Then,  $\phi^\ep=-2\ep\log u^\ep$, and \eqref{eq:viscousHJ} is equivalent to
\[
    -2\ep^2\Delta u^\ep-F(x) u^\ep=-c(\ep)u^\ep \qquad\text{on }\T^n.
\]
We thus have that $-c(\ep)$ is the principal eigenvalue of the Schr\"odinger operator $-2\ep^2\Delta - F$ on $\T^n$.
Normalize $\phi^\ep(0)=0$, which is equivalent to $u^\ep(0)=1$.
A major open question is whether the full sequence \(\{\phi^\ep\}\) converges as $\ep \to 0$.
See \cite[Problem 1]{JKM}, \cite[Zero-temperature selection problem]{A}, \cite[Problem 31]{AIM}, \cite[Remark 1]{AIPS}, \cite[Section 4]{GISY}, \cite[Section 6.6.2]{Le-Mitake-Tran}, \cite[Selection problem]{Soga}, \cite[Section 6.6]{Tran}, \cite{GaoLiu}, \cite[Questions 3-4]{Tran1}.

\begin{openproblem}\label{openquestion}
    For $\ep \in (0,1)$, let $\phi^\ep$ be the unique solution to \eqref{eq:viscousHJ} with $\phi^\ep(0)=0$.
    Is the full sequence \(\{\phi^\ep\}\) convergent in $C(\T^n)$ as $\ep \to 0$?
\end{openproblem}

{\bf Motivation:} Open Problem~\ref{openquestion} is precisely a selection problem for the vanishing-viscosity process. If
\[
\lim_{\ep\to 0}\phi^\ep=\phi,
\]
then \(\phi\) is a viscosity solution of the ergodic, or cell, problem
\[
\frac{1}{2}|D\phi|^2+F(x)=c(0) \qquad \text{on } \T^n.
\]
See \cite{CIL} for the definition of viscosity solutions. This equation admits multiple viscosity solutions, even modulo additive constants, when \(F\) has more than one maximum point. 
Thus, when the full vanishing-viscosity limit exists, the process selects a distinguished limiting solution among all possible viscosity solutions. 
Such a selected limit is often called a {\it physical solution}; see \cite{AIPS}. 
This selection mechanism was first introduced in \cite{JKM} in the study of
physical solutions to the inviscid Burgers equation. In the steady one-dimensional case, they considered the periodically forced Burgers equation
\[
    v v_x = -F'(x) \qquad \text{on } \T,
\]
with prescribed momentum
\[
    \int_{\mathbb T} v\,dx = p.
\]
A natural way to select a physical solution is through the vanishing-viscosity
approximation
\[
    v^\ep v_x^\ep - \ep v_{xx}^\ep = -F'(x),
\]
and to study the limit as \(\ep\to 0\). Thus, the existence of the full
vanishing-viscosity limit is a central issue. In the case \(p=0\), this formulation is connected to \eqref{eq:viscousHJ} through the relation $v^\ep = \phi^\ep_x$.
\medskip

{\bf Positive results:} The full limit \(\lim_{\ep\to 0}\phi^\ep\) exists under the following assumptions:
\begin{itemize}
    \item \(F\) has finitely many maximum points \((x_i)_{1\leq i\leq m}\), all of which are non-degenerate.
    \item Among these maximum points, there is a unique point \(x_I\) that minimizes
    \[
        \sum_{j=1}^n \sqrt{-\lambda_j(x_i)}.
    \]
    Here \((\lambda_j(x_i))_{1\leq j\leq n}\) denote the eigenvalues of the Hessian \(D^2F(x_i)\).
\end{itemize}

\noindent In this setting, the one-dimensional case was first proved in \cite{JKM}. Suitable extensions to higher dimensions can be found in \cite{A,AIPS}; these works also reveal deep connections with the Aubry--Mather theory. We refer to \cite{Ev4} for applications to quantum mechanics, and to \cite{Yu} for related results when \(F\) has degenerate maximum points.
See also \cite{Soga,GaoLiu} for related results on vanishing-viscosity selection problems.

\medskip

In this paper, we show that the answer to Open Problem~\ref{openquestion} is negative, even in dimension one.

\medskip

\subsection{Settings and the main result}
From now on, we let $n=1$.
Let $a=1/2$.
Choose \(0<\rho<1/20\).  For \(0<r<\rho\), define
\[
    \omega(r):=\sin\bigl(\log\log(1/r)\bigr).
\]
For \(0<|x|<\rho\), set
\[
    g(x):=x^4\omega(|x|),
    \qquad g(0):=0.
\]
We have that \(g\) is \(C^3\) near \(0\), and
\[
    g(0)=g'(0)=g''(0)=g'''(0)=0.
\]
Choose \(A\not=0\) and then choose \(\rho>0\) sufficiently small so that
\[
    2|A|\rho^2\le \frac14.
\]
We choose \(V\in C^3(\T)\) such that
\begin{equation}\label{eq:Vnear0}
    V(x)=\frac12 x^2+A x^4\omega(|x|),
    \qquad |x|<\rho,
\end{equation}
and
\begin{equation}\label{eq:VnearA}
    V(x)=\frac12 (x-a)^2-A (x-a)^4\omega(|x-a|),
    \qquad |x-a|<\rho.
\end{equation}
The opposite signs \(\pm A\) in the construction play a crucial role. 
They ensure that the two local Dirichlet ground-state energies oscillate in opposite directions; see Lemma~\ref{lem:localE}.

Away from these two neighborhoods, we choose \(V\) smoothly so that
\begin{equation}\label{eq:Vpost}
    V(x)>0\qquad\text{for }x\notin\{0,a\}.
\end{equation}
This is possible because, for \(|x|<\rho\),
\[
    |A x^4\omega(|x|)|\le |A|\rho^2 x^2\le \frac18 x^2,
\]
and similarly near \(a\).  
Thus, the quadratic term dominates the oscillatory quartic perturbation.
Set
\[
    F:=-V.
\]
Then \(F\in C^3(\T)\), \(\max_\T F=0\), and the only maxima of \(F\) are $0$ and $1/2$.
Moreover,
\[
    F''(0)=F''(a)=-1.
\]
Since $\max_{\T} F=0$, the limiting ergodic (cell problem) is
\begin{equation}\label{eq:cell-0}
    \frac{1}{2}|\phi'|^2 + F(x) = c(0)=0 \qquad \text{on } \T.
\end{equation}
Here, $c(0)=0$ thanks to the inf-max formula (see \cite[Chapter 4]{Tran}).

\begin{theorem}\label{thm:main}
    Assume $n=1$ and let $F=-V$, where $V$ satisfies \eqref{eq:Vnear0}--\eqref{eq:Vpost}.
    For $\ep \in (0,1)$, let $\phi^\ep$ be the unique solution to \eqref{eq:viscousHJ} with $\phi^\ep(0)=0$.
    Then, the full sequence $\{\phi^\ep\}$ does not converge in $C(\T)$ as $\ep \to 0$.
\end{theorem}

We note that our example does not contradict the one-dimensional full-convergence result of \cite{JKM}. 
Indeed, in \cite{JKM}, the relevant maxima of \(F\) are assumed to be nondegenerate and to have distinct second derivatives, whereas in our construction, the two maxima have the same second derivative.

\subsection*{Notations}
Throughout the paper, all intervals are understood in the periodic sense.
If a function $h:\R \to \R$ is periodic, we can think of $h$ as a function from $\T$ to $\R$ as well, and vice versa.
In this paper, we switch freely between the two
interpretations.

\subsection*{Organization of the paper}
In Section \ref{sec:Dirichlet ground states}, we study the local Dirichlet ground-state energies at $0$ and $a=1/2$.
Section \ref{sec:localization lemma} is devoted to a localization lemma if there is an $\ep^2$-energy gap.
Finally, Theorem \ref{thm:main} is proved in Section \ref{sec:main thm}.

\subsection*{Acknowledgment}
The authors used ChatGPT during the development of this work, including for suggesting proof strategies and assisting with  calculations. All mathematical statements, proofs, and verifications were subsequently completed and rigorously checked by the authors, who take full responsibility for the results. We appreciate ChatGPT for its support.

\section{Local Dirichlet ground-state energies}\label{sec:Dirichlet ground states}
Since \(F=-V\), we can write the principal eigenvalue problem as
\begin{equation}\label{eq:Schrodinger}
    P_\ep u^\ep=E(\ep)u^\ep \qquad\text{on }\T,
\end{equation}
where
\[
    P_\ep:=-2\ep^2\frac{d^2}{dx^2}+V,
    \qquad
    E(\ep):=-c(\ep).
\]
The function \(u^\ep\) is the positive principal eigenfunction of \(P_\ep\) with $u^\ep(0)=1$. It is well known that 
\begin{equation}\label{eq:basic-bound-c-ep}
    |E(\ep)|= |c(\ep)|=|c(\ep)-c(0)|\leq C\ep.
\end{equation}
Since \(V\geq 0\) and $V\not \equiv 0$, the Rayleigh characterization gives \(E(\ep)>0\). 
For the upper bound, choose a smooth cutoff function \(\xi\) supported in a small neighborhood $\sim O(\sqrt{\ep})$ of \(0\), with \(\xi\equiv1\) near \(0\), and use
\[
    v_\ep(x)=\xi(x)\exp\left(-\frac{x^2}{4\ep}\right)
\]
as a trial function in the Rayleigh quotient for \(P_\ep\). 
This gives \(E(\ep)\leq C\ep\). 
Thus, for $\ep\in (0,1)$,
\[
0< E(\ep)\leq C\ep.
\]
See \cite{Ev4, Yu, TuZhang} for other proofs of \eqref{eq:basic-bound-c-ep} and further properties of $E(\ep)$.
\medskip

Let \(E_0^D(\ep)\) be the first Dirichlet eigenvalue of \(P_\ep\) on
\((-\rho,\rho)\).  Let \(E_a^D(\ep)\) be the first Dirichlet eigenvalue of
\(P_\ep\) on \((a-\rho,a+\rho)\).
By the Rayleigh characterization, we have that
\begin{equation}\label{eq:basic-Reileigh}
    E(\ep) \leq \min\{E_0^D(\ep),E_a^D(\ep)\}.
\end{equation}

The following lemma is a standard harmonic-approximation computation for the local Dirichlet ground-state energies near the two nondegenerate wells, adapted here to the oscillatory quartic perturbations in \(V\).

\begin{lemma}\label{lem:localE}
As \(\ep\to0\),
\begin{align}
    E_0^D(\ep)
    &=\ep+3A\omega(\sqrt\ep)\ep^2+o(\ep^2),\label{eq:E0}\\
    E_a^D(\ep)
    &=\ep-3A\omega(\sqrt\ep)\ep^2+o(\ep^2).\label{eq:Ea}
\end{align}
\end{lemma}

\begin{proof}
We prove \eqref{eq:E0}; the proof of \eqref{eq:Ea} is identical with the sign
of \(A\) reversed.

In the well near \(0\), set
\[
    x=\sqrt\ep z.
\]
The operator becomes
\[
    -2\ep^2\frac{d^2}{dx^2}+\frac12x^2+A x^4\omega(|x|)
    =\ep\left(-2\frac{d^2}{dz^2}+\frac12 z^2
      +\ep A z^4\omega(\sqrt\ep |z|)\right).
\]
The rescaled interval is
\[
    I_\ep:=\left(-\frac{\rho}{\sqrt\ep},\frac{\rho}{\sqrt\ep}\right).
\]
Denote by \(e_0^D(\ep)=E_0^D(\ep)/\ep\) the first Dirichlet eigenvalue of
\begin{equation}\label{eq:eigen-L-ep}
    L_\ep:=-2\frac{d^2}{dz^2}+\frac12 z^2
      +\ep A z^4\omega(\sqrt\ep |z|)
    \qquad\text{in }I_\ep.
\end{equation}
Because \(|\omega|\le1\) and \(2|A|\rho^2\le1/4\), for \(|z|\le\rho/\sqrt\ep\),
\[
    \left|\ep A z^4\omega(\sqrt\ep |z|)\right|
    \le |A|\rho^2 z^2\le\frac18 z^2.
\]
Hence
\begin{equation}\label{eq:coerciveRescaled}
    L_\ep\ge -2\frac{d^2}{dz^2}+\frac38 z^2
\end{equation}
in the sense of quadratic forms.  This uniform coercivity implies that the
first eigenfunction is localized on bounded \(z\)-scales, uniformly in \(\ep\).
We now compute the first correction.

Let
\[
    H_0:=-2\frac{d^2}{dz^2}+\frac12 z^2
    \qquad\text{on }L^2(\R).
\]
Its normalized ground state is
\[
    \psi_0(z)=(2\pi)^{-1/4}e^{-z^2/4}
\]
with unit ground energy.  
The next eigenvalue is separated from \(1\) by
a fixed positive gap.

For the upper bound, take \(\psi_0\), cut it off smoothly inside \(I_\ep\), and
use it as a trial function in the Rayleigh characterization.  
Since the cutoff occurs at the distance \(\rho/(2\sqrt\ep)\), the cutoff error is \(O(e^{-\kappa/\ep})\).  Thus
\[
    e_0^D(\ep)
    \le 1+
    \ep A\int_\R z^4\omega(\sqrt\ep |z|)\psi_0(z)^2 \,dz
    +o(\ep).
\]

For the lower bound, let \(\psi_\ep\) be the normalized positive first
Dirichlet eigenfunction of \(L_\ep\) on \(I_\ep\).  The upper bound just obtained
and the coercivity \eqref{eq:coerciveRescaled} give
\begin{equation}\label{eq:energy-bound}
    \int_{I_\ep}\left(|\psi_\ep'|^2+z^2\psi_\ep^2\right)\,dz\le C.
\end{equation}
Multiply the eigenvalue equation $L_\varepsilon \psi_\varepsilon = e_0^D(\varepsilon)\psi_\varepsilon$  by
\(e^{\theta z^2}\psi_\ep\) and integrate over $I_\ep$. For small \(\theta>0\),  using the Cauchy-Schwarz inequality, we get
\begin{equation}\label{eq:bound-e-theta-z}
    \int_{I_\ep}e^{\theta z^2}\left(\frac{1}{2}|\psi_\ep'|^2+\frac{1}{4}z^2\psi_\ep^2\right)\,dz\le 2    \int_{I_\ep}e^{\theta z^2}\psi_\ep^2\,dz.
\end{equation}
On the region \(|z|\ge 4\), we have
\[
    2e^{\theta z^2}\psi_\ep^2
    \leq \frac{1}{8}e^{\theta z^2}z^2\psi_\ep^2.
\]
Hence the right-hand side of \eqref{eq:bound-e-theta-z} can be absorbed into the \(z^2\psi_\ep^2\)-term on the left outside \([-4,4]\), while the contribution on \([-4,4]\) is bounded by $C$ thanks to \eqref{eq:energy-bound}. 
Therefore,
\[
    \int_{I_\ep \setminus [-4,4]}e^{\theta z^2}z^2\psi_\ep(z)^2\, dz
    \le C.
\]
This implies
\begin{equation}\label{eq:uniformGaussianTail}
    \int_{I_\ep \setminus [-R,R]}\left(1+z^4\right)\psi_\ep(z)^2\, dz
    \le C e^{-cR^2}
    \qquad\text{for }4\le R\le \frac{\rho}{2\sqrt\ep},
\end{equation}
with constants $c,C>0$ independent of \(\ep\). 

By the energy bound \eqref{eq:energy-bound}, the family
\(\{\psi_\ep\}\), after extension by zero outside \(I_\ep\), is uniformly
bounded in \(H^1_{\mathrm{loc}}(\mathbb R)\). Together with the uniform tail
bound \eqref{eq:uniformGaussianTail}, this implies precompactness in
\(L^2(\mathbb R)\). Let \(\psi\) be the \(L^2\)-limit of a convergent
subsequence. Since \(e_0^D(\ep)\to 1\), the perturbation is \(o(1)\) on every
fixed bounded interval, and the tails are uniformly controlled by
\eqref{eq:uniformGaussianTail}, the limit \(\psi\) has unit \(L^2\)-norm and
attains the ground-state energy of
\[
    H_0=-2\frac{d^2}{dz^2}+{1\over 2}z^2
    \quad\text{on }L^2(\mathbb R).
\]
By the uniqueness and positivity of the ground state of \(H_0\), we have
\(\psi=\psi_0\). Hence, the whole family satisfies
\[
    \psi_\ep\to\psi_0
    \qquad\text{in }L^2(\mathbb R),
\]
after extending \(\psi_\ep\) by zero outside \(I_\ep\).  
Furthermore,
\[
\psi_\ep \rightharpoonup \psi_0 \qquad \text{in } H^1(\R).
\]
Using the uniform tail bound \eqref{eq:uniformGaussianTail} again, we also have
\begin{equation}\label{eq:z4Convergence}
    \int_{\R} z^4\omega(\sqrt\ep |z|)\psi_\ep(z)^2\,dz
    =
    \int_{\R} z^4\omega(\sqrt\ep |z|)\psi_0(z)^2\,dz+o(1).
\end{equation}
Therefore, we derive that
\begin{align*}
       e_0^D(\ep)
    &=\int_{\R}\left(2|\psi_\ep'|^2+\frac{1}{2}z^2\psi_\ep^2\right)\,dz+
    \ep A\int_\R z^4\omega(\sqrt\ep |z|)\psi_\ep(z)^2\,dz\\
    &\ge 1+
    \ep A\int_\R z^4\omega(\sqrt\ep |z|)\psi_0(z)^2\,dz
    +o(\ep). 
\end{align*}
The second inequality follows from the fact that \(1\) is the principal eigenvalue of \(H_0\), together with \eqref{eq:z4Convergence}.

Combining the upper and lower bounds yields
\begin{equation}\label{eq:preDominated}
    e_0^D(\ep)
    =1+
    \ep A\int_\R z^4\omega(\sqrt\ep |z|)\psi_0(z)^2\,dz
    +o(\ep).
\end{equation}

It remains to evaluate the integral.  For fixed \(z\ne0\),
\[
    \log\log\frac{1}{\sqrt\ep |z|}
    =\log\left(\log\frac{1}{\sqrt\ep}+\log\frac1{|z|}\right)
    =\log\log\frac1{\sqrt\ep}+o(1).
\]
Hence,
\[
    \omega(\sqrt\ep |z|)-\omega(\sqrt\ep)\to0
    \qquad\text{for fixed }z\ne0.
\]
Since \(|\omega|\le1\) and \(z^4\psi_0(z)^2\in L^1(\R)\), the dominated convergence theorem gives
\[
    \lim_{\ep \to 0}\int_\R \left(\omega(\sqrt\ep |z|)-\omega(\sqrt\ep)\right)z^4\psi_0(z)^2\,dz
    =0.
\]
Thus,
\[
    \int_\R z^4\omega(\sqrt\ep |z|)\psi_0(z)^2\,dz
    =\omega(\sqrt\ep)
      \int_\R z^4\psi_0(z)^2\,dz+o(1).
\]
As the density \(\psi_0^2\) is the centered Gaussian of variance \(1\),
\[
    \int_\R z^4\psi_0(z)^2\,dz=3.
\]
Thus,
\[
    e_0^D(\ep)=1+3A\omega(\sqrt\ep)\ep+o(\ep),
\]
and multiplying by \(\ep\) proves \eqref{eq:E0}.
\end{proof}

\section{A localization lemma}\label{sec:localization lemma}
Let \(U_\ep\in C(\T)\) be the positive principal eigenfunction of \(P_\ep\), normalized by
\[
    \int_\T U_\ep^2\,d x=1.
\]

We use the following localization lemma. 
It is a standard semiclassical localization statement: an energy gap between competing wells forces the ground state to concentrate near the lower well. 
For classical results in this direction, we refer the reader to \cite{Ag,Simon}.

\begin{lemma}[Localization from an \(\ep^2\)-energy gap]\label{lem:localization}
Fix \(C_0>0\).
There are constants \(\sigma>0\), \(C>0\), and \(\ep_0>0\) with the following
property.

If, along a sequence \(\ep\to0\),
\begin{equation}\label{eq:gap0}
    E_0^D(\ep)\le E_a^D(\ep)-C_0\ep^2,
\end{equation}
then, for \(0<\ep<\ep_0\) along that sequence,
\begin{equation}\label{eq:ratio0a}
    \frac{U_\ep(a)}{U_\ep(0)}\le C e^{-\sigma/\ep}.
\end{equation}
Consequently, with \(u^\ep=U_\ep/U_\ep(0)\) and
\(\phi^\ep=-2\ep\log u^\ep\),
\begin{equation}\label{eq:positiveAtA}
    \liminf_{\ep\to0}\phi^\ep(a)>0
\end{equation}
along that sequence.

Similarly, if
\begin{equation}\label{eq:gapa}
    E_a^D(\ep)\le E_0^D(\ep)-C_0\ep^2,
\end{equation}
then
\begin{equation}\label{eq:ratioa0}
    \frac{U_\ep(0)}{U_\ep(a)}\le C e^{-\sigma/\ep},
\end{equation}
and
\begin{equation}\label{eq:negativeAtA}
    \limsup_{\ep\to0}\phi^\ep(a)<0
\end{equation}
along that sequence.
\end{lemma}

\begin{proof}
We prove the case \eqref{eq:gap0}; the proof of \eqref{eq:gapa} is identical
with the roles of \(0\) and \(a\) exchanged.
We divide the proof into several steps.

\medskip
\noindent\textbf{Step 1: Exponential smallness in the barrier region.}

We first use a standard Agmon-type localization estimate: since
\(E(\varepsilon)=O(\varepsilon)\) while \(V\) is uniformly positive away
from the wells \(0\) and \(a\), the normalized principal eigenfunction is
exponentially small in the forbidden region:
\begin{equation}\label{eq:barrierMassSimple}
    \int_{\T\backslash O_r} U_\ep^2\,d x\le C e^{-s_r/\ep}.
\end{equation}
Here, $O_r= (-r,r)\cup (a-r,a+r)$ for $r\in (0,1/8)$ and $s_r$ is a positive constant depending on $r$ and $V$.

For the reader's convenience, we include a short proof in our one-dimensional setting.

Since
\(V>0\) on \(\mathbb T\setminus\{0,a\}\), there exists \(\nu>0\) depending on $r$ and $V$ such that
\[
    V\ge \nu \qquad \text{on $\T\backslash O_{r/2}$}.
\]
Choose a Lipschitz function \(\Phi\ge 0\) on \(\mathbb T\) and a small positive number $s$ such that
\begin{itemize}
    \item \(\Phi=0\) in $O_{r/2}$;
    \item  \(\Phi\ge s>0\) in  $\T\backslash O_r$;
    \item $\frac12|\Phi'|^2\le \frac{\nu}{4}$ in $\T\backslash O_{r/2}$.
\end{itemize}

Recall that \(U_\varepsilon\) satisfies
\[
    -2\varepsilon^2 U_\varepsilon''+VU_\varepsilon
    =E(\varepsilon)U_\varepsilon,
    \qquad \|U_\varepsilon\|_{L^2(\mathbb T)}=1.
\]
Testing this equation with \(e^{\Phi/\varepsilon}U_\varepsilon\), and completing
the square, gives the Agmon identity:
\[
    2\varepsilon^2
    \int_{\mathbb T}
    \left|
        \left(e^{\Phi/(2\varepsilon)}U_\varepsilon\right)'
    \right|^2 dx
    +
    \int_{\mathbb T}
    \left(
        V-E(\varepsilon)-\frac12|\Phi'|^2
    \right)
    e^{\Phi/\varepsilon}U_\varepsilon^2\,dx
    =0.
\]
Since \(|E(\varepsilon)|\le C\varepsilon\), for sufficiently small
\(\varepsilon\) we have
\[
    V-E(\varepsilon)-\frac12|\Phi'|^2
    \ge \frac{\nu}{2} \quad \text{on $\T\backslash O_{r/2}$}.
\]
Accordingly,  
\begin{align*}
 \frac{\nu}{2}
    \int_{\T\backslash O_{r/2}}
    e^{\Phi/\varepsilon}U_\varepsilon^2\,dx
    &\le -\int_{O_{r/2}}
    \left(
        V-E(\varepsilon)-\frac12|\Phi'|^2
    \right)
    e^{\Phi/\varepsilon}U_\varepsilon^2\,dx\\
    &=-\int_{O_{r/2}}
    \left(
        V-E(\varepsilon)
    \right)U_\varepsilon^2\,dx\\    
   & \leq E(\varepsilon)\int_{O_{r/2}} U_\varepsilon^2\,dx\le C\varepsilon.
\end{align*}
Since \(\Phi\ge s\) on \(\T\backslash O_r\), we obtain
\[
    \int_{\T\backslash O_r}U_\varepsilon^2\,dx
    \le C e^{-s/\varepsilon},
\]
after decreasing \(s>0\) if necessary. Hence (\ref{eq:barrierMassSimple}) holds. 

\medskip
\noindent\textbf{Step 2: Localization and small mass in the higher well.} Let $K=\T\backslash O_{\rho/3}$.  Choose three smooth cutoff functions \(\chi_0,\chi_a,\chi_b\) such that
\[
    \chi_0^2+\chi_a^2+\chi_b^2=1,
\]
\(\chi_0\) is supported in \((-\rho,\rho)\) and \(\chi_0\equiv 1\) on $(-\rho/4,\rho/4)$, \(\chi_a\) is supported in
\((a-\rho,a+\rho)\) and $\chi_a\equiv 1$ on \((a-\rho/4,a+\rho/4)\), \(\chi_b\) is supported in $K$. Moreover, the derivatives \(\chi_i'\) for $i=0,a,b$ are supported in  \(K\).  
See Figure \ref{fig:cutoff}.
\begin{center}
\begin{figure}[htp!]
\begin{tikzpicture}[x=1.15cm,y=3.8cm]

\draw[->, thick] (0,0) -- (10.4,0) node[right] {$x$};
\draw[->, thick] (0,0) -- (0,1.15);


\draw (0.02,1) -- (-0.02,1) node[left] {$1$};
\draw (0.02,0) -- (-0.02,0) node[left] {$0$};

\foreach \x in {2,3,7,8}
{
  \draw[dashed, gray] (\x,0) -- (\x,1.05);
}

\draw[very thick, blue] (0,1) -- (2,1);
\draw[very thick, blue, domain=2:3, samples=200, smooth, variable=\x]
  plot ({\x},{cos(90*(6*(\x-2)^5 - 15*(\x-2)^4 + 10*(\x-2)^3))});
\draw[very thick, blue] (3,0) -- (10,0);

\draw[very thick, red] (0,0) -- (2,0);
\draw[very thick, red, domain=2:3, samples=200, smooth, variable=\x]
  plot ({\x},{sin(90*(6*(\x-2)^5 - 15*(\x-2)^4 + 10*(\x-2)^3))});
\draw[very thick, red] (3,1) -- (7,1);
\draw[very thick, red, domain=7:8, samples=200, smooth, variable=\x]
  plot ({\x},{cos(90*(6*(\x-7)^5 - 15*(\x-7)^4 + 10*(\x-7)^3))});
\draw[very thick, red] (8,0) -- (10,0);

\draw[very thick, teal!70!black] (0,0) -- (7,0);
\draw[very thick, teal!70!black, domain=7:8, samples=200, smooth, variable=\x]
  plot ({\x},{sin(90*(6*(\x-7)^5 - 15*(\x-7)^4 + 10*(\x-7)^3))});
\draw[very thick, teal!70!black] (8,1) -- (10,1);

\node[blue] at (1,1.08) {$\chi_0$};
\node[red] at (5,1.08) {$\chi_b$};
\node[teal!70!black] at (9,1.08) {$\chi_a$};

\end{tikzpicture}
\caption{Graphs of $\chi_0, \chi_a, \chi_b$}
\label{fig:cutoff}
\end{figure}
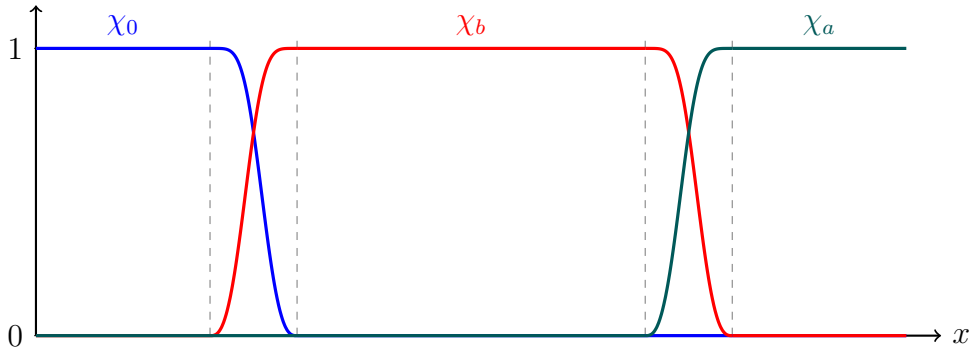
    
\end{center}
We have
\begin{equation}\label{eq:IMS}
\begin{aligned}
    \int_\T \left(2\ep^2 |U_\ep'|^2+VU_\ep^2\right)\,d x
    &=
    \sum_{\theta\in\{0,a,b\}}
    \int_\T\left(2\ep^2 |(\chi_\theta U_\ep)'|^2
      +V\chi_\theta^2U_\ep^2\right)\,d x \\
    &\qquad
    -2\ep^2\sum_{\theta\in\{0,a,b\}}
      \int_\T |\chi_\theta'|^2U_\ep^2\,d x .
\end{aligned}
\end{equation}
Because the derivatives of the cutoffs are supported in \(K\), the last term
is \(O(e^{-s/\ep})\) by \eqref{eq:barrierMassSimple}.  Also,
\(\chi_0U_\ep\in H_0^1((-\rho,\rho))\) and
\(\chi_aU_\ep\in H_0^1((a-\rho,a+\rho))\).  Hence
\[
    \int_{\T}\left(2\ep^2 |(\chi_0 U_\ep)'|^2+V\chi_0^2U_\ep^2\right)\,dx
    \ge E_0^D(\ep)\int_{\T} \chi_0^2U_\ep^2\,dx,
\]
and
\[
    \int_{\T}\left(2\ep^2 |(\chi_a U_\ep)'|^2+V\chi_a^2U_\ep^2\right)\,dx
    \ge E_a^D(\ep)\int_{\T} \chi_a^2U_\ep^2\,dx.
\]
The remaining \(\chi_b\)-term is nonnegative.  Therefore
\begin{equation}\label{eq:lowerMassEnergy}
    E(\ep)
    \ge
    E_0^D(\ep)m_0(\ep)+E_a^D(\ep)m_a(\ep)-C e^{-s/\ep},
\end{equation}
where
\[
    m_0(\ep):=\int_{\T}\chi_0^2U_\ep^2\,dx,
    \qquad
    m_a(\ep):=\int_{\T}\chi_a^2U_\ep^2\,dx.
\]
Moreover, by \eqref{eq:barrierMassSimple},
\begin{equation}\label{eq:m0-mb-bound}
    m_0(\ep)+m_a(\ep)=1-\int_\T\chi_b^2U_\ep^2\,dx=1-\int_K\chi_b^2U_\ep^2\,dx=1+O(e^{-s/\ep}).
\end{equation}
Using \eqref{eq:basic-Reileigh}, \eqref{eq:gap0}, and \eqref{eq:lowerMassEnergy}, we get
\[
    E_0^D(\ep)
    \ge E(\ep)
    \ge
    E_0^D(\ep)(m_0(\ep)+m_a(\ep))+C_0\ep^2m_a(\ep)-C e^{-s/\ep}.
\]
Thus
\[
    C_0\ep^2m_a(\ep)
    \le C e^{-s/\ep}.
\]
Absorbing the polynomial factor \(\ep^{-2}\) into the exponential, we obtain,
for a smaller \(s>0\),
\begin{equation}\label{eq:massInHigherWell}
    m_a(\ep)\le C e^{-s/\ep}.
\end{equation}
Since \(\chi_a\equiv1\) on \((a-\rho/4,a+\rho/4)\),
\begin{equation}\label{eq:L2nearA}
    \int_{|x-a|<\rho/4} U_\ep^2\,d x\le C e^{-s/\ep}.
\end{equation}
Thanks to \eqref{eq:m0-mb-bound} and \eqref{eq:massInHigherWell}, $m_0(\ep)=1+O(e^{-s/\ep})$. Since $\chi_0\equiv 1$ on $(-\rho/4,\rho/4)$, combining with \eqref{eq:barrierMassSimple}, we have that
\begin{equation}\label{eq:L2near0Large}
    \int_{|x|<\rho/4} U_\ep^2\,d x\ge \frac12
\end{equation}
for all sufficiently small \(\ep\).

\medskip
\noindent\textbf{Step 3: From \(L^2\)-localization to pointwise comparison.}
We first estimate \(U_\ep(a)\) from above.  On the interval
\(|x-a|<2\sqrt\ep\), write
\[
    x=a+\sqrt\ep z,
    \qquad
    W_\ep(z):=U_\ep(a+\sqrt\ep z).
\]
For \(|z|\le2\), the eigenvalue equation becomes
\[
    -2W_\ep''(z)+\frac{V(a+\sqrt\ep z)}{\ep}W_\ep(z)
    =\frac{E(\ep)}{\ep}W_\ep(z).
\]
Since \(V(a+\sqrt\ep z)/\ep\) and \(E(\ep)/\ep\) are bounded for \(|z|\le2\),
we have
\[
    -W_\ep''(z)=q(z) W_\ep(z)
    \qquad\text{for }|z|\le2,
\]
where $\|q\|_{L^\infty((-2,2))} \leq C$.
Let $\xi$ be a smooth cutoff function such that $\xi=1$ in $(-1,1)$ and $\xi$ is supported in $(-2,2)$.
Multiply the above by $\xi^2 W_\ep$ and integrate to imply
\[
    \int_{-2}^2|(\xi W_\ep)'|^2\,dz\le C\int_{-2}^{2} W_\ep(z)^2\,d z.
\]
Therefore, by the fundamental theorem of calculus and the H\"older inequality,
\[
    |W_\ep(0)|^2=\left(\int_0^2 (\xi W_\ep)'\,dz\right)^2\le 2 \int_{0}^2|(\xi W_\ep)'|^2\,dz\leq  C\int_{-2}^{2} W_\ep(z)^2\,d z.
\]
Changing variables back to \(x\), and using \eqref{eq:L2nearA}, we get
\begin{equation}\label{eq:pointUpperA}
    U_\ep(a)^2
    \le C\ep^{-1/2}\int_{|x-a|<2\sqrt\ep}U_\ep(x)^2\,d x
    \le C\ep^{-1/2}e^{-s/\ep}.
\end{equation}
Thus, decreasing \(s\) once more,
\begin{equation}\label{eq:pointUpperA2}
    U_\ep(a)
    \le C e^{-s/\ep}.
\end{equation}

We next estimate \(U_\ep(0)\) from below.  Since
\[
    E(\ep)=\int_\T\left(2\ep^2|U_\ep'|^2+VU_\ep^2\right)\,d x\le C\ep,
\]
and near \(0\) we have \(V(x)\ge c x^2\), it follows that
\[
    \int_{|x|<\rho/4} x^2U_\ep(x)^2\,d x\le C\ep.
\]
Combining this with \eqref{eq:L2near0Large}, we can choose a fixed large
number \(R>1\), independent of \(\ep\), such that
\begin{equation}\label{eq:massSqrtScale}
    \int_{|x|<R\sqrt\ep}U_\ep(x)^2\,d x\ge \frac14.
\end{equation}
Rescale \(x=\sqrt\ep z\), \(W_\ep(z):=U_\ep(\sqrt\ep z)\).  On \(|z|\le 2R\),
\(W_\ep\) satisfies a second-order linear equation 
\[
    -2W_\ep''(z)+\frac{V(\sqrt\ep z)}{\ep}W_\ep(z)
    =\frac{E(\ep)}{\ep}W_\ep(z).
\]
Note that \(0 \leq V(\sqrt\ep z)/\ep \leq CR^2\) and \(|E(\ep)/\ep| \leq C\) for \(|z|\le 2R\).
We have
\[
    -W_\ep''(z)=q(z) W_\ep(z)
    \qquad\text{for }|z|\le 2R,
\]
where $\|q\|_{L^\infty((-2R,2R))}\leq C(1+R^2)$.
As $W_\ep>0$, the one-dimensional Harnack inequality gives
\[
    \sup_{|z|\le R} W_\ep(z)
    \le C_R W_\ep(0).
\]
Using \eqref{eq:massSqrtScale},
\[
    \frac14
    \le \sqrt\ep\int_{|z|<R} W_\ep(z)^2\,d z
    \le C_R\sqrt\ep\, W_\ep(0)^2.
\]
Therefore
\begin{equation}\label{eq:pointLower0}
    U_\ep(0)
    \ge c_R\ep^{-1/4}.
\end{equation}
Combining \eqref{eq:pointUpperA2} and \eqref{eq:pointLower0}, and absorbing the
polynomial factor into the exponential, yields
\[
    \frac{U_\ep(a)}{U_\ep(0)}\le C e^{-\sigma/\ep}
\]
for some \(\sigma>0\).  
Hence,
\[
    \liminf_{\ep \to 0}\phi^\ep(a)
    =\liminf_{\ep \to 0}\left(-2\ep\log\frac{U_\ep(a)}{U_\ep(0)}\right)
    \ge 2\sigma.
\]
This proves \eqref{eq:positiveAtA}.

\end{proof}

\section{Proof of Theorem \ref{thm:main}} \label{sec:main thm}

\begin{proof}[Proof of Theorem \ref{thm:main}]
    Define
\begin{align*}
    \ep_n^+&:=\exp\left(-2\exp\left(\frac\pi2+2\pi n\right)\right),\\
    \ep_n^-&:=\exp\left(-2\exp\left(\frac{3\pi}{2}+2\pi n\right)\right).
\end{align*}
Then
\[
    \omega(\sqrt{\ep_n^+})=1,
    \qquad
    \omega(\sqrt{\ep_n^-})=-1.
\]
Without loss of generality, assume $A>0$. By Lemma \ref{lem:localE},
\begin{align*}
    E_0^D(\ep_n^+) &=\ep_n^+ +3A(\ep_n^+)^2+o((\ep_n^+)^2),\\
    E_a^D(\ep_n^+) &=\ep_n^+ -3A(\ep_n^+)^2+o((\ep_n^+)^2).
\end{align*}
Thus, for all large \(n\),
\[
    E_a^D(\ep_n^+)
    \le E_0^D(\ep_n^+)-A(\ep_n^+)^2.
\]
By Lemma \ref{lem:localization}, with the normalization \(\phi^{\ep}(0)=0\),
\[
    \limsup_{n\to\infty}\phi^{\ep_n^+}(a)<0.
\]

Similarly,
\begin{align*}
    E_0^D(\ep_n^-) &=\ep_n^- -3A(\ep_n^-)^2+o((\ep_n^-)^2),\\
    E_a^D(\ep_n^-) &=\ep_n^- +3A(\ep_n^-)^2+o((\ep_n^-)^2).
\end{align*}
Therefore, for all large \(n\),
\[
    E_0^D(\ep_n^-)
    \le E_a^D(\ep_n^-)-A(\ep_n^-)^2.
\]
Again by Lemma \ref{lem:localization},
\[
    \liminf_{n\to\infty}\phi^{\ep_n^-}(a)>0.
\]
Thus, along one sequence the values \(\phi^\varepsilon(a)\) are eventually bounded away from zero on the negative side, while along another sequence they are bounded away from zero on the positive side. Consequently, the normalized family \(\{\phi^\varepsilon\}\) cannot converge in \(C(\T)\).
\end{proof}

\begin{remark}[Finite regularity]
The same construction of Theorem \ref{thm:main} can be adapted for $F$ in any finite regularity class $C^k$.
Indeed, replacing the perturbation \(x^4\omega(|x|)\) by
\(x^{2m}\omega(|x|)\) for $m\ge 2$ gives a potential of class \(C^{2m-1}\).  The same
harmonic-approximation computation yields
\[
    E_0^D(\ep)
    =
    \ep+c_m A\omega(\sqrt\ep)\ep^m+o(\ep^m),
\]
and
\[
    E_a^D(\ep)
    =
    \ep-c_m A\omega(\sqrt\ep)\ep^m+o(\ep^m),
\]
where
\[
    c_m=\int_\R z^{2m}\psi_0(z)^2\,dz=(2m-1)!!.
\]
Since this polynomial-size gap dominates the exponentially small tunneling error through the barrier, the same argument gives nonconvergence of the full
vanishing-viscosity family.  Thus, for every finite \(k\), one can choose
\(m\) large enough to obtain an example with \(F\in C^k(\T)\).
\end{remark}

We have the following consequence.
\begin{corollary}[Nonconvergence of ground-state measures]
Let \(V\) be as in Theorem \ref{thm:main}.  Let \(U_\ep>0\) be the
\(L^2(\T)\)-normalized principal eigenfunction of
\[
    P_\ep=-2\ep^2\frac{d^2}{dx^2}+V.
\]
Set
\[
    \mu_\ep:=U_\ep^2\,dx.
\]
Then the family \(\{\mu_\ep\}_{\ep>0}\) does not converge weakly in the sense of measures as \(\ep\to0\).  
More precisely, for the sequences \(\ep_n^\pm\) defined in the
proof of Theorem \ref{thm:main} in section \ref{sec:main thm},
we have the following weak convergence in the sense of measures
\[
    \mu_{\ep_n^+}\rightharpoonup \delta_a,
    \qquad
    \mu_{\ep_n^-}\rightharpoonup \delta_0.
\]
\end{corollary}

\begin{proof}
    Along the sequence \(\ep_n^+\), Lemma \ref{lem:localE} gives
\[
    E_a^D(\ep_n^+)\le E_0^D(\ep_n^+)-A(\ep_n^+)^2
\]
for all large \(n\).  The proof of Lemma \ref{lem:localization}, with the
roles of \(0\) and \(a\) exchanged, implies that the \(L^2\)-mass of
\(U_{\ep_n^+}\) in the \(0\)-well and in the barrier region is exponentially
small.  Hence, the mass is concentrated in the \(a\)-well.

Moreover,
\[
    E(\ep)=\int_\T\left(2\ep^2|U_\ep'|^2+VU_\ep^2\right)\,dx\le C\ep.
\]
Since \(V(x)\ge c|x-a|^2\) around \(a\), it follows that, for every fixed
\(\eta>0\),
\[
    \int_{|x-a|>\eta}U_{\ep_n^+}^2\,dx\to0.
\]
Therefore \(\mu_{\ep_n^+}\rightharpoonup \delta_a\).

The proof along \(\ep_n^-\) is the same.  In that case \(E_0^D(\ep_n^-)\le
E_a^D(\ep_n^-)-A(\ep_n^-)^2\), so the mass concentrates at \(0\), and hence
\[
    \mu_{\ep_n^-}\rightharpoonup \delta_0.
\]
Thus \(\{\mu_\ep\}\) cannot have a weak limit.
\end{proof}

\begin{remark}
It is natural to ask a complementary quantitative question: in situations where the
full vanishing-viscosity limit
\[
    \phi^\ep \to \phi \qquad \text{in } C(\T^n)
\]
does exist, can one obtain a sharp convergence rate?

From the homogenization point of view, \eqref{eq:viscousHJ} is the cell problem
corresponding to \(P=0\) for the periodic homogenization, as \(\ep\to 0\), of the
viscous quadratic Hamilton--Jacobi equation
\[
    u_t^\ep-\ep\Delta u^\ep
    +\frac{1}{2}|D u^\ep|^2
    +F\left(\frac{x}{\ep}\right)=0.
\]
The corresponding sharp convergence rates, both globally and almost everywhere,
were obtained recently by the same authors in \cite{LTY-rate}. Methodologically,
both works exploit the special quadratic structure through the Hopf--Cole transform
and the associated spectral analysis of a linear Schr\"odinger operator, although with
different focuses. It would be interesting to investigate whether the quantitative
methods developed in \cite{LTY-rate} can also shed light on the convergence rate
\(\phi^\ep\to\phi\) in settings where the full vanishing-viscosity limit exists.
\end{remark}

\end{document}